\documentclass{article}
\usepackage{cite} 
\usepackage{tikz}
\usepackage{amsmath,amsthm, amsfonts}
\usepackage{color}
\usepackage[spanish,english]{babel}
\usepackage[latin1]{inputenc}
\usepackage{graphicx}
\usepackage{fancybox}
\usepackage{titlesec}
\usepackage{float}
\usepackage{enumerate}
\usepackage{multicol}
\usepackage{flushend}
\usepackage{pgfplots}
\usepackage[bookmarks]{hyperref}
\usepgfplotslibrary{polar}
\usepgflibrary{shapes.geometric}
\usetikzlibrary{calc}
\pgfplotsset{my style/.append style={axis x line=middle, axis y line=
middle, xlabel={$x$}, ylabel={$y$}, axis equal }}
\newtheorem{thm}{Theorem}[section]

\newtheorem{lem}[thm]{Lemma}
\newtheorem{prop}[thm]{Proposition}
\theoremstyle{definition}

\theoremstyle{remark}

\title{Relaxation  limit  for  Aw-Rascle  system}
\author{Richard Alexander de la Cruz\\
Juan Carlos Juajibioy Otero \\
Leonardo Rendon\\
    }
\date{Departamento de Matemáticas, Universidad Nacional de  Colombia\\
 Bogotá, 2013}

\begin{document}
\maketitle
\abstract{\noindent We study the relaxation limit  for  the Aw-Rascle system of  traffic  flow. For  this we apply the theory of invariant  regions and  the compensated  compactness method  to get global  existence of  Cauchy problem for a particular Aw-Rascle system  with  source, where the source is the  relaxation term, and we show the convergence of this  solutions  to the equilibrium  state  }  
\section{Introduction }\noindent
In \cite{Luref2} the author introduces the  system  
\begin{equation}\label{rascle}
\begin{cases}
\rho_t +(\rho v)_x=0,\\
(v+P(\rho))_t+v(v+P(\rho))_x=0,
\end{cases}
\end{equation}
as  a model  of second  order  of  traffic flow. It  was proposed by the author to remedy the deficiencies  of  second order model or car  traffic pointed in \cite{daganzo} by the  author. The system (\ref{rascle}) models a single  lane traffic where  the  functions $\rho(x,t)$ and $v(x,t)$  represent the density and the velocity of  cars on the road  way and $P(\rho)$ is a given function describing the anticipation of road  conditions in  front  of the drivers. In \cite{Luref2} the author solves the Riemann problem  for the  case in which  the vacuum appears  and the  case in which  the  vacuum  does not. Making the  change of  variable 
\[
w=v+P(\rho),
\] the system (\ref{rascle}) is  transformed in to the system
\begin{equation}\label{rascle2}
\begin{cases}
\rho_t +(\rho (w-P(\rho)))_x=0,\\
w_t+(w-P(\rho))w_x=0.
\end{cases}
\end{equation}
Multiplying  the  second equation  in (\ref{rascle2}) by $\rho$ we  have  the system
\begin{equation}\label{rascle3}
\begin{cases}
\rho_t+(\rho(w-P(\rho))=0,\\
(w\rho)_t+(w\rho(w-P(\rho)))_x=0.
\end{cases}
\end{equation}
Now making the substitution  $m=w\rho$, system (\ref{rascle3}) is transformed in to  system
\begin{equation}\label{rascle4}
\begin{cases}
\rho_t+(\rho\phi(\rho,m))=0,\\
(m)_t+(m\phi(\rho,m))_x=0.
\end{cases}
\end{equation}
where $\phi(\rho,m)=\frac{m}{\rho}-P(\rho)$, this is a system of  non symmetric Keyfitz-Kranzer type. In \cite{Luref1}, the  author, using  the Compensate Compactness Method, shows the  existence of global bounded solutions  for the Cauchy problem for the homogeneous system(\ref{rascle4}). In this paper we are concerned  with the Cauchy problem for the following  Aw-Rascle system 
\begin{equation}\label{awrascle}
\begin{cases}
\rho_t+(m -\rho P(\rho))_x=0,\\
m_t+(\frac{m^2}{\rho}+mP(\rho))_x=\frac{1}{\tau}(h(p)-m).
\end{cases}
\end{equation}
with  bounded measurable  initial  data 
\begin{equation}\label{awdata}
(\rho_(x,0),m(x,0))=(\rho_0(x)+\epsilon,m_0(x)),
\end{equation}
Where  the source  term $\frac{1}{\tau}(h(p)-m)$ is  called  relaxation  term, $\tau$ is the relaxation time  and $h(\rho)$
is an equilibrium  velocity. To see important issues in this model  see \cite{rascle1} and references therein. 
Let us put $F(\rho,m)=(\rho\phi(\rho,m),m\phi(\rho,m))$, with $\phi(\rho,m)=\frac{m}{\rho}-P(\rho)$ by direc calculations we have that the  eigenvalues and  corresponding eigenvectors of system (\ref{awrascle}) are given by
\begin{align}\label{eigen}
\lambda_1(\rho,m)&=\frac{m}{\rho}-P(\rho)-\rho P^{'}(\rho),\ r_1=(1,\frac{m}{\rho}),\\
\lambda_2(\rho,m)&=\frac{m}{\rho}-P(\rho),\  r_2=(2,\frac{m}{\rho}+\rho P^{'}(\rho)).
\end{align}
The  Riemann's invariants  are given  by
\begin{align}
W(\rho,m)&=\frac{m}{\rho},\ \ Z(\rho,m)=\frac{m}{\rho}-P(\rho).
\end{align}
Now as
\[
\nabla \lambda_1 \cdot r_1=-(2P^{'}(\rho)+\rho P^{''}(\rho)), \ \ \nabla \lambda_2\cdot r_2=0,
\]
we  see that  the the  second  wave family  is always linear  degenerate and the  behavior  of the  second family  wave 
depends to the  values  of  $\rho P(\rho)$. In  fact,  if  $\theta(\rho)=\rho P(\rho)$ concave or  convex then  the  second
family wave  is genuinely non linear, see \cite{shockrascle} to the case in which  the  two  families  wave  are linear  degenerate,
\section{The positive invarian  regions}
In this  section  we  show the  theorem  for invariant  regions  for  find a estimates  a  priori of the   parabolic system  (\ref{norelaxawrascle})\\[0.2in]
Let 
\begin{prop}\label{invariant regions}
Let $\overline{\mathcal{O}}\subset \Omega\subset \mathbb{R}^2$  be  a compact, convex  region  whose boundary
consists  of  a finite number  of level  curves $\gamma_j$ of  Riemann invariants, $\xi_j$, such  that 
\begin{equation}\label{geom}
(U-Y)\nabla\xi_j(U)>0 \ \ \text{for}\  U\in \gamma_j, \ \ Y\in \mathcal{O}.
\end{equation}
where $U=(u,v)$, $Y=(y_1,y_2)$.
If $U_0(x)\in \mathbf{L}^{\infty}(\mathbb{R})\times \mathbf{L}^{\infty}(\mathbb{R})$ and $U_0(x)\in K\subset \subset \mathcal{O}$ for all $x\in \mathbb{R}$, with $U_0(x)=(u_0(x),v_0(x))$ then
for  any $\epsilon >0$  the solution  of the system
\begin{equation}
\begin{cases}
u^{\epsilon}_t+f(u^{\epsilon},v^{\epsilon})_x=\epsilon u^{\epsilon}_{xx},\\
v^{\epsilon}_t+g(u^{\epsilon},v^{\epsilon})_x=\epsilon v^{\epsilon}_{xx},\\
\end{cases}
\end{equation}
with initial  data 
\begin{equation}
u^{\epsilon}(x,0)=u_0(x), \  \ v^{\epsilon}(x,0)=v_0(x),
\end{equation}
exists in $[0,\infty)\times \mathbb{R}$  and $(u^{\epsilon}(x,t),v^{\epsilon}(x,t))\in \overline{\mathcal{O}}$.
\end{prop} 
\begin{proof}
It is  sufficient to prove the  result  for $u_0(x),v_0(x) \in \mathbf{C}^{\infty}(\mathbb{R})$.  Let 
$U^{\epsilon,\delta}=(u^{\epsilon,\delta},v^{\epsilon,\delta})$ be the unique solution  of the  Cauchy problem
\begin{equation}\label{cauchy}
\begin{cases}
\partial_t U^{\epsilon,\delta}+\partial_x F(U^{\epsilon,\delta})=\epsilon\partial^2_xU^{\epsilon,\delta}-\delta\nabla  P(U^{\epsilon,\delta}),\\
U^{\epsilon,\delta}(x,0)=U_0(x)
\end{cases}
\end{equation}
where $F=(f,g)$ and $P(U)=|U-Y|^2$ for  some  fix $Y\in \mathcal{O}$.  If  we suppose  that $U^{\epsilon,\delta}\notin\mathcal{O}$ for all $(x,t)$, then there  exist some $t_0>0$ and $x_0$  such  as
\[
U^{\epsilon,\delta}(x_0,t_0) \in \partial \mathcal{O}.
\]
Since $\in \partial \mathcal{O}=\cup \gamma_j$,  $U^{\epsilon,\delta}(x_0,t_0)\in \gamma_j$ for  some $j$. Then, multiplying  by $\nabla \xi_j$ in (\ref{cauchy}) we have,
\begin{align}\label{eqRI0}
&\partial_t\xi(U^{\epsilon,\delta})+\lambda_i(U^{\epsilon,\delta})\partial_x\xi(U^{\epsilon,\delta})\\
&=\partial^2_x\xi(U^{\epsilon,\delta})-(\partial_xU^{\epsilon,\delta})^TH\xi(U^{\epsilon,\delta})(\partial_xU^{\epsilon,\delta})-\delta\nabla\xi(U^{\epsilon,\delta})\nabla P(U^{\epsilon,\delta}).\notag
\end{align}
Now by (\ref{geom}) we have that  
\begin{equation}\label{eqRI1}
(\partial_xU^{\epsilon,\delta})^TH\xi(U^{\epsilon,\delta})(\partial_xU^{\epsilon,\delta})\geq 0,
\end{equation}
and 
\begin{equation}\label{eqRI2}\delta\nabla\xi(U^{\epsilon,\delta})\nabla P(U^{\epsilon,\delta})=2(U^{\epsilon,\delta}-Y)\nabla\xi(U^{\epsilon,\delta})>0.
\end{equation}
The  characterization of $(x_0,t_0)$ implies 
\begin{equation}\label{eqRI3}
\partial_x\xi(U^{\epsilon,\delta}(x_0,t_0))=0,\ \  \\ \partial^2_x\xi(U^{\epsilon,\delta}(x_0,t_0)).
\end{equation}
Replacing (\ref{eqRI1}), (\ref{eqRI2}), (\ref{eqRI3})  in (\ref{eqRI0}), we have that
\[
\partial_t\xi(U^{\epsilon,\delta}(x_0,t_0))<0, 
\]  
which is a contradiction. Now  we  show that $U^{\epsilon,\delta}\to U^{\delta}$ as $\delta  \to 0$. For this
let $W^{\epsilon,\delta,\sigma}$ be the solution  of
\begin{equation}\label{eqRI4}
\partial_t W^{\epsilon,\delta,\sigma}+\partial_x(G^{\epsilon,\delta,\sigma}W^{\epsilon,\delta,\sigma})=\epsilon W^{\epsilon,\delta,\sigma}-\delta\nabla P(U^{\epsilon,\delta})+\sigma\nabla P(U^{\epsilon,\sigma}),
\end{equation}
where
\[
G^{\epsilon,\delta,\sigma}=\int_0^1DF(sU^{\epsilon,\delta}+(1-s)U^{\epsilon,\sigma})ds.
\]
Multiplying by $W^{\epsilon,\delta,\sigma}$ in (\ref{eqRI4}), and integrating  over $\mathbb{R}$,  we have
\begin{equation}\label{eqRI5}
\frac{d}{dt}\|W^{\epsilon,\delta,\sigma}(t)\|^2_{\mathbf{L}^2(\mathbb{R}^2)}\leq \frac{K}{\epsilon}\|W^{\epsilon,\delta,\sigma}(t)\|^2_{\mathbf{L}^2(\mathbb{R}^2)} +K(\delta+\sigma).
\end{equation} 
Then, integrating  respect  to  variable $t$  over interval $(0,t)$ we have that  
\begin{equation}
\|W^{\epsilon,\delta,\sigma}(t)\|^2_{\mathbf{L}^2(\mathbb{R}^2)}\leq K(\delta+\sigma)t+\int_0^1\frac{K}{\epsilon}\|W^{\epsilon,\delta,\sigma}(t)\|^2_{\mathbf{L}^2(\mathbb{R}^2)}dt.
\end{equation}
Finally by applying  Gronwall's inequality we obtain
\[
\|W^{\epsilon,\delta,\sigma}(t)\|^2_{\mathbf{L}^2(\mathbb{R}^2)}\leq K(\delta+\sigma)te^{\frac{K}{\epsilon}t}
\]
then, for $\sigma=0$ and $\delta \to 0$ we have  that $U^{\epsilon,\delta}\to U^{\epsilon}$ as $\delta \to 0$.
\end{proof}
\section{Relaxation limit}
Based in the Theory of Invariant  Regions  and  Compensated Compactness Method we can obtain the following  result.
\begin{thm}\label{teorem1}
Let $h(\rho)\in \mathbf{C}({\mathbb{R}})$. Suppose  that there  exists  a region
\[
\Sigma=\left\{(\rho,m):W(\rho,m)\leq C1, Z(\rho,m)\geq C2, \rho\geq 0\right\}
\]
were $C1>0$, $C_2>0$. Assume that $\Sigma$ is  shuch that   the curve $m=h(\rho)$ as $0\leq \rho<\rho_1$ and  the  initial  data (\ref{awrascle})
are inside  $\Sigma$ and $(\rho_1,m_1)$ is the intersection of the curve $W=C_1$  with $Z=C_2$. Then, for any fixed
$\epsilon>0$, $\tau>0$ the solution $(\rho^{\epsilon,\tau}(x,t),m^{\epsilon,\tau}(x,t))$ of the Cauchy problem (\ref{awrascle}), (\ref{awdata}) globally  exists and  satifies
\begin{equation}\label{resaw}
0\leq \rho^{\epsilon,\tau}(x,t)\leq M, \  0\leq m^{\epsilon,\tau}(x,t)\leq M, \ (x,t)\in[0,\infty)\times\mathbb{R}. 
\end{equation}
Moreover, if $\tau=o(\epsilon)$ as $\epsilon\to 0$ then there exists a subsequence $(\rho^{\epsilon,\tau},m^{\epsilon,\tau})$
converging a.e.  to $(\rho,m)$ as $\epsilon\to 0$, where $(\rho,m)$ is  the equilibrium state uniquely determined by
\begin{enumerate}[I.]
\item The  function $m(x,t)$ satisfies $m(x,t)=h(\rho(x,t))$ for  almost all $(x,t)\in[0,\infty)\times\mathbb{R}$.
\item The  function $\rho(x,t)$ is the $\mathbf{L}^{\infty}$ entropy solution  of the Cauchy problem
\[
\rho_t+(\rho h(\rho))_x=0,\ \rho(x,0)=\rho_0(x).
\]
\end{enumerate}
\end{thm}\noindent
The proof  of this  theorem is postponed for later, first we   collect  some preliminary estimates  in the  following lemmas. 
\begin{lem}\label{lema0}
Let $(\rho_{\epsilon},m_{\epsilon})$ be solutions of the system (\ref{awrascle}), with bounded measurable initial data (\ref{awdata}), and the condition S given in (ref) holds. Then $(\rho_{\epsilon},m_{\epsilon})$
is uniformly bounded in $\mathbf{L}^{\infty}$ with respect $\epsilon$ and $\tau$.
\end{lem}
\begin{proof}
First,  we show that the region 
\[
\Sigma=\left\{(\rho,m):W(\rho,m)\leq C1, Z(\rho,m)\geq C2, \rho\geq 0\right\}
\]
is invariant  for the parabolic system (see Figure 1)
\begin{equation}\label{norelaxawrascle}
\begin{cases}
\rho_t+(m -\rho P(\rho)))_x=\epsilon \rho_{xx},\\
m_t+(\frac{m^2}{\rho}+mP(\rho))_x=\epsilon m_{xx}.
\end{cases}
\end{equation}
 If $\gamma_1$ is given for $m(\rho)=C_1\rho$ and $\gamma_2$ is given
for $m(\rho)=\rho C_2+\rho P(\rho)$, it is easy to show that if $u=(\overline{\rho},\overline{m})\in \gamma_1$ and
$y=(\rho,m)\in \Sigma$ it then holds
\[
(u-y)\nabla W(u)>0
\]
and  if $u=(\overline{\rho},\overline{m})\in \gamma_2$ and
$y=(\rho,m)\in \Sigma$ then we have
\[
(u-y)\nabla W(u)>0.
\]
Using Proposition (\ref{invariant regions}), we have that $\Sigma$ is an invariant  region for (\ref{norelaxawrascle}).
\end{proof}
\begin{figure}[h!]\label{fig1}
\begin{center}
\begin{tikzpicture}
\draw (0,0)  (3,3);
\draw[very thick,<->] (3,0) node[below]{$\rho$} -- (0,0) --
(0,3) node[left]{$m$};
\tikzset{func/.style={thick,color=black!90}}
\draw[func,domain=0:2.5] plot [samples=200] (\x,{\x})node[right] {\small $W=C_1$};
\draw[func,domain=0:1.7] plot [samples=200] (\x,{1/5*\x^5})node[left] {\small $Z=C_2$};
\draw[thick] plot [smooth,tension=1.5] coordinates{(0,0)(0.75,0.5)(1.5,1.5)};
\draw (1.75,1)node[below] {\small $h(\rho)=m$};
\draw (0.5,1.2)node[below] {\small $\Sigma_1$};
\end{tikzpicture}
\end{center}
\caption{Riemann Invarian Region I}
\end{figure}
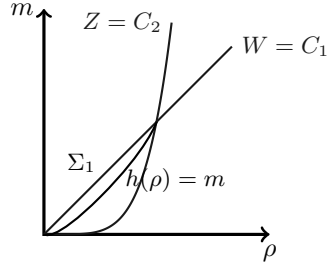
For the case in which the system \ref{awrascle} contains relaxation  term we  use the ideas of the  authors in \cite{relaxforpreassure} 
\begin{eqnarray}
R(y,s)&=\rho(\tau y, \tau s)\\
M(y,s)&=m(\tau y, \tau s),
\end{eqnarray}
the  system (\ref{awrascle}) is transformed into the system
\begin{equation}
\begin{cases}
R_s+\left(M+RP(R)\right)_y=\epsilon R_{yy},\\
M_s+\left(\frac{M^2}{R}-RP(R)\right)_{y}=h(R)-M+\epsilon M_{yy}
\end{cases}
\end{equation}
which does not depend on $\tau$, and taken $C_1=W(1,0)$, $C_2=Z(1,0)$ the  curves $M=Q(R)$, $W=C1$, $Z=C_2$ intersect
with  the $R$ axis at the same point in $R=0$, $R=1$ (see Figure 2). Using the stability conditions
\begin{equation}
\lambda_1(\rho,h(\rho))<h^{'}(\rho)<\lambda_1(\rho,h(\rho)),
\end{equation} it is  easy to show  that  the  vector $(0,h(R)-M)$ points inwards the  region $\Sigma_2$ and from \cite{positive} it follows that $\Sigma_2$ is an invariant  region.
\begin{figure}[h!]\label{fig1}
\begin{center}
\begin{tikzpicture}
\draw (0,0)  (4,4);
\draw[very thick,<->] (4,0) node[below]{$R$} -- (0,0) --
(0,3) node[left]{$M$};
\tikzset{func/.style={thick,color=black!90}}
\draw[func,domain=0:2.5] plot [samples=200] (\x,{0*\x})node[below] {\small $W=C_1^0$};
\draw[func,domain=0:1.7,rotate=-45] plot [samples=200] (\x,{1/5*(\x)^5})node[left] {\small $Z=C_2$};
\draw[thick,rotate=-45] plot [smooth,tension=1.5] coordinates{(0,0)(0.75,0.5)(1.5,1.5)};
\draw (0.5,0.5)node[below] {\small $\Sigma_2$};
\end{tikzpicture}
\end{center}
\caption{Riemann Invarian Regions II}
\end{figure}
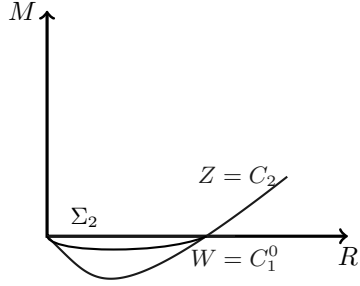

\begin{lem}\label{lema1}
If  the solutions  of (\ref{awrascle}), (\ref{awdata}) have an  a priori $\mathbf{L}^{\infty}$ bounds, and $h\in \mathbf{C}^2$, then $\epsilon (p_x)2$, $\epsilon(m_x)2$, $\frac{(h(\rho)-m)}{\tau}$ are  bounded in $\mathbf{L}^1_{loc}$ on the case
that $\tau=o(\epsilon)(\epsilon \to 0)$, when $\epsilon\to 0$.
\end{lem}
\begin{proof}
Let $Q(\rho,m)=\frac{m^2}{2}-h(\rho)m+\frac{C_1 \rho^2}{2}$, since $(\rho,m)$ is bounded we can  choose $C_1$ such that
\[
Q_{\rho \rho}(\rho_x)^2+2Q_{\rho m}\rho_x m_x +Q_{m m}(m_x)^2\geq C_2(\rho^2_x+m^2_x).
\]
Multiplying the system (\ref{awrascle}) by $(Q_{\rho},Q_m)$  we have
\begin{align}\label{eqor1}
&Q(\rho,m)_t+Q_{\rho}(\rho,m)(\rho\phi(\rho,m))_x+Q_{m}(\rho,m)(m\phi(\rho,m))_x\\
&\leq \epsilon\left(Q_{xx}-C_2(\rho^2_x+m^2_x)\right).\notag
\end{align}
Adding terms and applying  the mean value  theorem in the  $m$ variable  to the functions
$\phi_1(\rho,m)=\rho\phi(\rho,m)$ and $Q(\rho,m)$ we have that 
\begin{equation}\label{eq1}
Q_{\rho}(\rho,m)(\rho\phi(\rho,m))_x=T_1+T_2+T_3,
\end{equation}
where
\begin{align*}
T_1&=\left(Q_{\rho}(\rho,m)\left(\phi_1(\rho,m)-\phi_1(\rho,h\rho)\right)+\int^{\rho}\phi_1(s,h(s)\frac{d}{ds}\phi_1(s,h(s))ds\right)_x,\\
T_2&=-\left(Q_{\rho \rho}(\rho,m)\rho^2_x+Q_{\rho m}(\rho,m)m^2_x\right)\phi_{1m}(\rho,\alpha_1)(m-h(\rho)) \ \text{and}\\
T_3&=Q_{\rho m}(\rho,\beta_1)(m-h(\rho))\left(\phi_{1 \rho}(\rho,h(\rho))+\phi_{1m}(\rho,h(\rho))h^{'}(\rho)\right)\rho_x.
\end{align*}
Putting $\phi_2(\rho,m)=m\phi(\rho,m)$ and  proceeding  as  above we have  that
\begin{equation}\label{eq2}
Q_{m}(\rho,m)(m\phi(\rho,m))_x=\overline{T_1}+\overline{T_2}+\overline{T_3},
\end{equation}
where
\begin{align*}
\widehat{T_1}&=\left(Q_{m}(\rho,m)\left(\phi_2(\rho,m)-\phi_2(\rho,h\rho)\right)+\int^{\rho}\phi_2(s,h(s)\frac{d}{ds}\phi_2(s,h(s))ds\right)_x,\\
\widehat{T_2}&=-\left(Q_{m \rho}(\rho,m)\rho^2_x+Q_{mm}(\rho,m)m^2_x\right)\phi_{2m}(\rho,\alpha_1)(m-h(\rho))\ \text{and}\\
\widehat{T_3}&=Q_{m m}(\rho,\beta_1)(m-h(\rho))\left(\phi_{2 \rho}(\rho,h(\rho))+\phi_{2m}(\rho,h(\rho))h^{'}(\rho)\right)\rho_x.
\end{align*}
Now replacing the values of $Q$ and  $\phi_1$ in $T_2$ we have that
\[T_2=-\left((-h^{''}(\rho)m+C_1)\rho^2_x-h^{'}(\rho)m^2_x\right)(m-h\rho),\] then
\[
|T_2|=C|\rho^2_x+m^2_x||h(\rho)-m|,
\]
where $C=max(|h^{''}(\rho)m-C_1|,|h^{'}( \rho)|)$. Using  the Young's  $\delta$-inequality 
\begin{equation}\label{ineLu}
|ab|\leq\delta a^2+\frac{b^2}{4\delta},
\end{equation}
we have  that 
\begin{equation}\label{ineqT1}
|T_2|\leq C_2(\delta)\tau(\rho^2_x+m^2_x)+\delta\frac{(m-h(\rho))^2}{\tau}.
\end{equation}
For $\widehat{T_2}=-\left(-h^{'}(\rho)\rho^2_x+m^2_x\right)(2\frac{\alpha_2}{\rho}-P(\rho))(m-h(\rho))$, using (\ref{ineLu}) we have 
\begin{equation}\label{ineTb1}
|\widehat{T_2}|\leq \widehat{C_2}(\delta)\tau(\rho^2_x+m^2_x)+\delta\frac{(h(\rho)-m)^2}{4\tau}.
\end{equation}
For $T_3$ and $\widehat{T_3}$ we have
\begin{align}
|T_3|&\leq \delta\frac{(h(\rho)-m)^2}{\tau}+C_3(\delta)\rho^2_x\ \text{and}\\
|\widehat{T_3}|&\leq \delta\frac{(h(\rho)-m)^2}{\tau}+\widehat{C_3}(\delta)\rho^2_x. 
\end{align}
Let us introducing the  following
\begin{align}\label{eq1aux}
A&= C_2(\delta)\tau(\rho^2_x+m^2_x)+\delta\frac{(m-h(\rho))^2}{\tau},\\
\widehat{A}&= \widehat{C_2}(\delta)\tau(\rho^2_x+m^2_x)+\delta\frac{(m-h(\rho))^2}{\tau},\\
B&=\delta\frac{(h(\rho)-m)^2}{\tau}+C_3(\delta)\rho^2_x,\\
\widehat{B}&=\delta\frac{(h(\rho)-m)^2}{\tau}+\widehat{C_3}(\delta)\rho^2_x,
\end{align}
and $R(\rho,m)=T_1+\widehat{T_2}$. Then  substituting  (\ref{eq1}), (\ref{eq2}) in (\ref{eqor1}) and  using
(\ref{ineTb1}-\ref{eq1aux}), and $\delta=\frac{1}{8}$  we have
\begin{equation}\label{ineqdef}
Q(\rho,m)_t+R(\rho,m)_t +(\epsilon C_2-\tau C_4)(\rho^2_x+m^2_x)+\frac{(h(\rho)-m)^2}{2\tau}\leq\epsilon Q(\rho,m)_{xx}.
\end{equation}
For  $\epsilon $ sufficiently  small we can choose $C_2$, $C_4$ such  that $C_4\tau\leq (C_2-T)\epsilon$ for $T>0$. Let  $K$ be a compact  subset of $\mathbb{R}\times \mathbb{R^{+}}$  and  $\Phi(x,t)\in \mathbf{D}(\mathbb{R}\times \mathbb{R^{+}})$, sucha  that $\Phi=1$ in $K$, $0\leq \Phi\leq 1$. Then, multiplying  (\ref{ineqdef}) by $\Phi(x,t)$ and integrating by parts  we have
\begin{equation}
\int_{\mathbb{R}\times \mathbb{R^{+}}}\left(2T\epsilon(\rho^2_x+m^2_x)\Phi+\frac{(h(\rho)-m)^2}{\tau}\Phi\right) dxdt\leq M(\Phi).
\end{equation}
\end{proof}
\begin{lem}\label{lema2}
If $(\eta(\rho),q(\rho))$ is  any  entropy-entropy flux pair for  the scalar equation
\begin{equation}
\rho_+(\rho \phi(\rho,h(\rho)))_x=0,
\end{equation}
then 
\[
\eta(\rho)_t+q(\rho)_x 
\]
is compact in $H^{-1}(\mathbb{R}\times \mathbb{R^{+}})$.
\end{lem}
\begin{proof}
Adding $\psi(\rho)=\rho\phi(\rho,h(\rho))$ in the first  equation of (\ref{awrascle}) we have
\begin{equation}\label{flux1}
\rho_t+\psi(\rho)_x=\epsilon \rho_{xx}+(\rho\phi(\rho,h(\rho))-\rho\phi(\rho,m)),
\end{equation}
and multiplying  by $\eta^{'}8q)$ in (\ref{flux1}) we have that
\begin{align*}
\eta(\rho)_t+q(\rho)_x&=\epsilon \eta(\rho)_{xx}-\epsilon \eta^{2}(\rho)\rho_{xx}\\
&+\left(\eta^{'}(\rho)(\psi(\rho)-\rho\phi(\rho,m))\right)-\eta^{2}(\rho)(\psi(\rho)-\rho\phi(\rho,m))\rho_x.
\end{align*}
Let $A^{\epsilon}=\epsilon \eta(\rho)_{xx}-\epsilon \eta^{2}(\rho)\rho_{xx}$,  and 
$B^{\epsilon}=\left(\eta^{'}(\rho)(\psi(\rho)-\rho\phi(\rho,m))\right)-\eta^{2}(\rho)(\psi(\rho)-\rho\phi(\rho,m))\rho_x$,
we  state  that $A^{\epsilon}\in H^{-1}_{Loc}(\mathbb{R}\times \mathbb{R^{+}})$ and $B^{\epsilon}$ is bounded in $M(\mathbb{R}\times \mathbb{R^{+}})$, then  for Murat's lemma  we have that 
\[
\eta(\rho)_t+q(\rho)_x 
\]
is compact in $\mathbf{H}^{-1}_{Loc}(\mathbb{R}\times \mathbb{R^{+}})$.
For the first  affirmation see \cite{serre},  we  show that $B^{\epsilon}$ is  bounded in $L^{1}_{Loc}$. Applying
the  mean value   theorem in the second variable  to the function $\phi(\rho,m)$ in $[h(\rho),m]$ and using Lemma \ref{lema1} we have
\begin{align*}
&\int_{\Omega}\eta^{''}(\rho)(\psi(\rho)-\rho\phi(\rho,m))\rho_xdxdt\\
&\leq M\int_{\Omega}(h(\rho)-m)\rho_xdxdt
\leq M\left(\int_{\Omega}\frac{(h(\rho)-m)^2}{\tau}dxdt\right)^{\frac{1}{2}}\left(\int_{\Omega}\tau \rho^2_xdxdt\right)^{\frac{1}{2}},
\end{align*}
and 
\begin{align*}
&\left| \int_{\Omega}(\eta^{'}(\rho)(\psi(\rho)-\rho\phi(\rho,m)))_x\Phi(x,t)dxdt|\right| =\left| \int_{\Omega}(\eta^{'}(\rho)(\psi(\rho)-\rho\phi(\rho,m)))\Phi(x,t)_xdxdt|\right|\\
&\leq  M\left(\int_{\Omega}\frac{(h(\rho)-m)^2}{\tau}dxdt\right)^{\frac{1}{2}}\left(\int_{\Omega}\tau \Phi^2_xdxdt\right)^{\frac{1}{2}}.
\end{align*}
\end{proof}
Now we   prove   Theorem \ref{teorem1}. By  the  Lemma \ref{lema0} we have the a priori bounds (\ref{resaw}), and we also have that  there is  a subsequence  of $(\rho^{\epsilon},m^{\epsilon})$ such  as
\begin{equation}\label{weaklim}
\rho(x,t)=w^{*}-\lim \rho^{\epsilon}(x,t), \ \ m(x,t)=w^{*}-\lim m^{\epsilon}(x,t) 
\end{equation}
Let us introduce the following
\begin{align}\label{entropi}
\eta_1(\theta)&=\theta-k,\\ 
q_1(\theta)&=\psi(\theta)-\psi(k),\\ 
\eta_2(\theta)&=\psi(\theta)-\psi(k) \ \text{and}\\ 
q_2(\theta)&=\int_{\rho}^{\rho^{\epsilon}}(\psi^{'}(s))^2ds. \label{e100}
\end{align}
Then  by the  weak  convergence of determinant \cite{BookLu} page 15, we have that
\begin{equation}\label{weakdet}
\overline{\eta_1(\rho^{\epsilon})q_2(\rho^{\epsilon})}-\overline{\eta_2(\rho^{\epsilon})q_1(\rho^{\epsilon})}=
\overline{\eta_1(\rho^{\epsilon})}\overline{q_1(\rho^{\epsilon})}-\overline{\eta_2(\rho^{\epsilon})}\overline{q_1(\rho^{\epsilon})},
\end{equation}
by  direct  calculations, replacing  $\rho^{\epsilon}$ in \ref{entropi}-\ref{e100}   we have that
\begin{align*}
&\overline{(\rho^{\epsilon}-\rho)\int_{\rho}^{\rho^{\epsilon}}(\psi(\rho))^2-(\psi(\rho^{\epsilon})-\psi(\rho))^2}+\left(\overline{\psi(\rho^{\epsilon})-\psi(\rho)}\right)^2=\\
&\overline{(\rho^{\epsilon}-\rho)}\overline{\int_{\rho}^{\rho^{\epsilon}}(\psi^{'}(s))^2ds},
\end{align*}
an  since by (\ref{weaklim})
\[
\overline{(\rho^{\epsilon}-\rho)}\overline{\int_{\rho}^{\rho^{\epsilon}}(\psi^{'}(s))^2ds}=0.
\]
we have  that
\begin{align}
\overline{\psi(\rho^{\epsilon})}&=\psi(\rho),\\
\overline{(\rho^{\epsilon}-\rho)\int_{\rho}^{\rho^{\epsilon}}(\psi^{'}(s))^2ds-(\psi(\rho^{\epsilon})-\psi(\rho))^2}&=0
\end{align}
Now, using Minty's argument \cite{lions} or  arguments of author in \cite{zhin} it's finished the proof of the Theorem 3.1  
\section{Acknowledgments}
We  would like to thanks Professor Juan Carlos  Galvis   by his observation, and  many valuable suggestions  and to the
professor Yun-guang Lu  by his suggestion this problem.
\bibliographystyle{amsplain}
\bibliography{biblio}
\end{document}